\documentclass[a4paper,10pt]{article}
 \usepackage[english]{babel}
  \usepackage{amsmath,amsfonts,amssymb}
   \usepackage{paralist}
    \usepackage{graphics} 
  \usepackage{epsfig,psfrag} 

  \usepackage{hyperref} 

\newtheorem{theorem}{Theorem}[section]
\newtheorem{corollary}[theorem]{Corollary}
\newtheorem{lemma}[theorem]{Lemma}
\newtheorem{proposition}[theorem]{Proposition}

\newtheorem{definition}[theorem]{Definition}
\newtheorem{remark}{Remark}
\newtheorem{counter}{Counterexample}

\def\N{\mathbb{N}}
\def\Z{\mathbb{Z}}
\def\Q{\mathbb{Q}}
\def\R{\mathbb{R}}

\let\e=\varepsilon
\let\vp=\varphi
\let\t=\tilde
\let\ol=\overline
\let\ul=\underline
\let\grad=\nabla
\let\.=\cdot
\let\0=\emptyset
\let\mc=\mathcal

\def\fa{\forall\;}

\def\solose{\ \Rightarrow\ }

\def\meno{\,\backslash\,}
\def\pp{,\cdots,}

\def\O{\Omega}

\def\dist{\text{\rm dist}}

\def\norma#1{\|#1\|_\infty}

\def\eq#1{{\rm(\ref{eq:#1})}}
\def\thm#1{Theorem \ref{thm:#1}}
\def\square{\hbox{$\sqcap\kern-7pt\sqcup$}}

\def\seq#1{(#1_n)_{n\in\N}}
\def\limn{\lim_{n\to\infty}}

\def\pe{principal eigenvalue}
\def\MP{maximum principle}
\def\SMP{strong maximum principle}

\def\lp{\lambda'_1(-\ll,\R^N)}

\newenvironment{proof}{\emph{Proof.}}{\hfill$\Box$\vskip7pt}
\newenvironment{proofof}[1]{\emph{Proof of #1.}}{\hfill$\Box$\vskip7pt}
\newenvironment{formula}[1]{\begin{equation}\label{eq:#1}}
                       {\end{equation}\noindent}
\def\formulaI#1{\begin{formula}{#1}}
\def\formulaF{\end{formula}\noindent}
\newenvironment{nota}[1]{\marginpar{\footnotesize{#1}}}

\def\immpiccola#1#2#3#4#5{
\begin{figure}[htbp]
\begin{center}
\includegraphics[width=#3cm]{#1}
#2
\end{center}
\caption{#4} \label{#5}
\end{figure}}

\def\pe{principal eigenvalue}
\def\lp{\lambda_p(-L)}

\title{\bf{Liouville type results for periodic and almost
periodic linear operators}}
\author{Luca Rossi\thanks{
EHESS, CAMS, 54 Boulevard Raspail, F-75006, Paris, France}}

\begin{document}
\maketitle

\begin{abstract}
This paper is concerned with some extensions of the classical
Liouville theorem for bounded harmonic functions to solutions of
more general equations. We deal with entire solutions of periodic
and almost periodic parabolic equations including the elliptic
framework as a particular case. We derive a Liouville type result
for periodic operators as a consequence of a result for operators
periodic in just one variable, which is new even in the elliptic
case. More precisely, we show that if $c\leq0$ and $a_{ij},\ b_i,\
c,\ f$ are periodic in the same space direction or in time, with
the same period, then any bounded solution $u$ of
$$\partial_t u-a_{ij}(x,t)\partial_{ij}u-b_i(x,t)\partial_iu-c(x,t)u=f(x,t),\quad
x\in\R^N,\ t\in\R,$$
is periodic in that direction or in time. 
We then derive the following Liouville
type result: if $c\leq0,\ f\equiv0$
and $a_{ij},\ b_i,\ c$ are periodic in all the space/time variables, with
the same periods, then the
space of bounded solutions of the above equation has at most
dimension one. In the case of the equation $\partial_t u-Lu=f(x,t)$,
with $L$ periodic elliptic operator independent of $t$, the hypothesis
$c\leq0$ can be weaken
by requiring that the periodic \pe\ $\lambda_p$ of $-L$ is nonnegative.
Instead, the periodicity assumption
cannot be relaxed, because we explicitly exhibit an almost
periodic function $b$ such that the space of bounded solutions of
$u''+b(x)u'=0$ in $\R$ has dimension 2, and it is generated by the constant
solution and a non-almost periodic solution.

The above counter-example leads us to consider the following
problem: under which conditions are bounded solutions necessarily
almost periodic? We show that a sufficient condition in the case of
the equation
$\partial_t u-Lu=f(x,t)$ is:
$f$ almost periodic and $L$ periodic with $\lambda_p\geq0$.

Finally, we consider problems in general
periodic domains under either Dirichlet or Robin boundary
conditions. We prove analogous properties as in the whole space, together
with some existence and uniqueness results for entire solutions.





\end{abstract}

\bigskip
{\small \noindent MSC Primary: 35B10, 35B15;
Secondary: 34C27, 35B05.\\
Keywords: Linear elliptic operator, Liouville theorem, periodic
solutions, almost periodic solutions, maximum principle.}


\section{Introduction}


\subsection{Statement of the main results}\label{sec:main}




We study the properties of bounded
{\bf entire solutions} - that is, solutions for all times -
of the parabolic equation
\formulaI{P=0}
Pu=0,\quad x\in\R^N,\ t\in\R, \formulaF
with
$$Pu=\partial_t u-a_{ij}(x,t)\partial_{ij}u-b_i(x,t)\partial_i u-c(x,t)u$$
(the convention is adopted for summation from $1$ to $N$ on repeated indices,
and $\partial_i,\ \partial_{ij}$ denote the space-directional derivatives).
We want to find in particular conditions
under which the Liouville property (LP) holds. In analogy with the
classical result for harmonic functions, we say that the LP
holds if the space of bounded solutions has at most
dimension one.

We will sometimes restrict our analysis to time-independent operators, that
we write as $P=\partial_t-L$, with $L$ general elliptic operator
in non-divergence form:
$$Lu=a_{ij}(x)\partial_{ij}u+b_i(x)\partial_i u+c(x)u.$$
The associated stationary solutions satisfy the elliptic equation
$-Lu=0$ in $\R^N$.

Our assumptions on the coefficients are: $a_{ij},b_i\in
L^\infty(\R^N\times\R) \cap UC(\R^N\times\R)$ (where $UC$ stands for
uniformly continuous), $c\in
L^\infty(\R^N\times\R)$ and the matrix field $(a_{ij})_{i,j}$ is
symmetric and uniformly elliptic, that is,
$$\fa t\in\R,\ x,\xi\in\R^N,\qquad\ul a|\xi|^2\leq a_{ij}(x,t)\xi_i\xi_j\leq \ol
a|\xi|^2,$$ for some constants $0<\ul a\leq\ol a$.
Let us mention that, in the case of elliptic equations, the uniform
continuity of
the $b_i$ can be dropped. We will sometimes denote the generic space/time
point $(x,t)\in\R^N\times\R$ by $X\in\R^{N+1}$.

We consider in particular operators with {\em periodic} and {\em almost
periodic} coefficients.
We say that a function $\phi:\R^{N+1}\to\R$ is periodic in the $m$-th
variable, $m\in\{1\pp N+1\}$, with period $l_m>0$, if
$\phi(X+l_me_m)=\phi(X)$ for $X\in\R^{N+1}$, where $(e_1,\cdots,e_{N+1})$
denotes the canonical basis of $\R^{N+1}$. If $\phi$ is periodic in all
the variables we simply say that it is periodic, with period
$(l_1,\cdots,l_{N+1})$. A linear operator is said to be periodic
(resp.~periodic in the $m$-th variable) if all its coefficients
are periodic (resp.~periodic in the $m$-th variable) {\bf with the same
period}.

The crucial step to prove the LP consists in
showing that the periodicity of the operator $P$ and of the function $f$
is inherited by bounded solutions
\footnote{ for us, a solution of a parabolic equation such as \eq{P=f} is
 a function $u\in L^p_{loc}(\R^{N+1})$, for all $p>1$, such that
 $\partial_t u,\partial_i u,\partial_{ij}u
\in L^p_{loc}(\R^{N+1})$ and such that \eq{P=f} holds a.~e. We use an
analogous definition for elliptic equations.
In the sequel, we will omit to write
a.~e.~for properties concerning measurable functions, and we will simply
denote by $\inf$ and $\sup$ the $\mathrm{ess }\inf$ and $\mathrm{ess }\sup$.}
of
\formulaI{P=f}
Pu=f(x,t),\quad x\in\R^N,\ t\in\R. \formulaF
Unless otherwise specified, the function $f$ is only assumed to be
measurable.

\begin{theorem}\label{thm:1per}
Let $u$ be a bounded solution of \eq{P=f}, with $P,\ f$ periodic
in the $m$-th variable, with the same period $l_m$, and with
$c\leq0$. Then, $u$ is periodic in the $m$-th variable, with
period $l_m$.
\end{theorem}

From the above result it follows in particular that if $P$ and $f$ do
not depend on $t$ and $c\leq0$, then all bounded solutions of
\eq{P=f} are stationary, that is, constant in time.
Another consequence of \thm{1per} is that if
$P$ and $f$ are periodic (in all the variables) with the same
period, then all bounded solutions are periodic. In particular,
they admit global maximum and minimum and then the strong maximum
principle implies the LP.

\begin{corollary}\label{cor:liouville}
Let $u$ be a bounded solution of \eq{P=0} 
with $P$ periodic and $c\leq0$. Then, two
possibilities occur:
\begin{itemize}

\item[{\rm 1)}] $c\equiv0$ and $u$ is constant;

\item[{\rm 2)}] $c\not\equiv0$ and $u\equiv0$.

\end{itemize}
\end{corollary}

Clearly, without the assumption $c\leq0$ the LP no
longer holds in general, even in the case of constant
coefficients. As an example, the space of solutions of $-u''+u=0$
in $\R$ is generated by $u_1=\sin x$ and $u_2=\cos x$. However,
if $P=\partial_t-L$,
condition $c\leq0$ in Corollary \ref{cor:liouville} is not
necessary and can be relaxed by
requiring that the {\em periodic \pe} of $-L$ in $\R^N$ is nonnegative
(cf.~\thm{per} below).
Henceforth, $\lambda_p(-L)$ will always
stand for the periodic \pe\ of $-L$ in $\R^N$ and $\varphi_p$
for the associated principal eigenfunction (see Section
\ref{sec:per} for the definitions).

\begin{theorem}\label{thm:per}
Let $P=\partial_t-L$, with $L$ periodic with period $(l_1\pp l_N)$,
and let $f$ be periodic with period $(l_1\pp l_{N+1})$.
If $u$ is a bounded solution of \eq{P=f} we have that:
\begin{itemize}

\item[{\rm(i)}] if $\lambda_p(-L)\geq0$ then $u$ is periodic, with
period $(l_1\pp l_{N+1})$;

\item[{\rm(ii)}] if $\lambda_p(-L)=0$ and either $f\leq0$ or $f\geq0$
then $u\equiv k\varphi_p$, for some $k\in\R$, and $f\equiv0$;

\item[{\rm(iii)}] if $\lambda_p(-L)>0$ and $f\equiv0$
then $u\equiv 0$.

\end{itemize}
\end{theorem}

In the particular case of stationary solutions, that is,
solutions of the elliptic equation $Lu=0$,
statements (ii) and (iii) of \thm{per} are contained in
\cite{KP} (see the next section for further details).
\thm{per} part (iii) immediately implies the uniqueness of bounded solutions to
\eq{P=f}. The existence result is also derived 
(cf.~Corollary \ref{cor:!} below).


We next consider the problem of the validity of the LP
if we relax the periodicity assumptions on $a_{ij},\ b_i,\
c$ and $f$. A natural generalization of periodic functions of a
single real variable are {\em almost periodic} functions,
introduced by Bohr \cite{ap}. This notion can be readily extended to
functions of several variables through a characterization of {\bf
continuous} almost periodic functions due to Bochner \cite{apChar}.

\begin{definition}\label{def:ap}
We say that a function $\phi\in C(\R^{N+1})$ is almost periodic
(a.~p.) if from any arbitrary sequence $\seq{X}$ in $\R^{N+1}$ can be
extracted a subsequence $(X_{n_k})_{k\in\N}$ such that
$(\phi(X+X_{n_k}))_{k\in\N}$ converges uniformly in $X\in\R^{N+1}$.
\end{definition}

It is straightforward to check that continuous periodic functions
are a.~p.~ (this is no longer true if we drop the continuity
assumption). We say that a linear operator is a.~p.~if its
coefficients are a.~p.

By explicitly constructing a counterexample, we show that the
Liouville type result of Corollary \ref{cor:liouville} does not
hold in general - even for elliptic equations - if we require the
operator to be only a.~p.

\begin{counter}\label{c-e}{\rm
There exists an a.~p.~function $b:\R\to\R$ such that the space of
bounded solutions to \formulaI{c-e} u''+b(x)u'=0\quad\text{in }\R
\formulaF has dimension 2, and it is generated by the function
$u_1\equiv1$ and a function $u_2$ which is not a.~p.}
\end{counter}

This also shows that bounded solutions of a.~p.~equations with
nonpositive zero order term may not be a.~p., in contrast with
what happens for the periodicity (cf.~\thm{1per}). Actually, the
function $b$ in Counterexample \ref{c-e} is {\em limit periodic},
that is, it is the uniform limit of a sequence of continuous
periodic functions (see Definition \ref{def:lp} below). Limit
periodic functions are a subset of a.~p.~functions because, as it
is easily seen from Definition \ref{def:ap}, the space of
a.~p.~functions is closed with respect to the $L^\infty$ norm (see
e.~g.~\cite{amerio-prouse}, \cite{Fink}).
\\

Next, we look for sufficient conditions under which all bounded
solutions of \eq{P=f} are necessarily a.~p. Under the additional assumption
$c\in C(\R^N)$,
we derive the following.


\begin{theorem}\label{thm:ap}
Let $P=\partial_t-L$ with $L$ periodic and let $f$ be a.~p.
If $u$ is a bounded solution of \eq{P=f} we have that:
\begin{itemize}

\item[{\rm(i)}] if $\lambda_p(-L)\geq0$ then $u$ is a.~p.;

\item[{\rm(ii)}] if $\lambda_p(-L)=0$ and either $f\leq0$ or $f\geq0$
then $u\equiv k\varphi_p$, for some $k\in\R$, and $f\equiv0$.

\end{itemize}
\end{theorem}

In the above statement, we require $c$ to be
continuous because, in the proof, we will make use of the fact that
it is in particular a.~p. Actually, using some weak compactness arguments,
one can check that the continuity assumptions on $c$ could be removed.

Lastly, we prove analogous results to Theorems \ref{thm:1per}, 
\ref{thm:per} and Corollary \ref{cor:liouville}
for equations in general periodic domains, under either
Dirichlet or Robin boundary conditions. 
The analogue of \thm{1per} holds, in the case of Dirichlet
boundary conditions, for domains periodic just
in the direction $x_m$, whereas under Robin conditions we are able
to prove the result only for domains periodic in all the directions.
The Liouville type result in the Dirichlet case is stronger than in
the whole space 
(Corollary \ref{cor:liouville}) and it is actually a uniqueness result.
An existence result is also obtained using a sub and supersolution method.
In some of the statements for general domains, we require that 
the coefficients of the operator are H\"older continuous because we
need some gradient estimates near the boundary.


\subsection{A brief survey of the related literature}

Starting from the end of the 50's, the classical Liouville theorem has
been improved
to the self-adjoint elliptic equation
\formulaI{s-a}
\partial_i(a_{ij}(x)\partial_j u)+c(x)u=0\quad\text{in }\R^N.
\formulaF In the case $c\equiv0$ (without any periodicity assumption on
$a_{ij}$), the LP follows directly from the estimate
on the oscillation of weak solutions proved by De Giorgi in the
celebrated paper
\cite{dG}. Another classical way to derive the LP is by applying
the Harnack inequality in the balls $B_r$, provided that one can bound
the constants uniformly with respect to $r$. This has been done by
Gilbarg and Serrin \cite{GS} for the equation
$a_{ij}(x)\partial_{ij}u=0$, with $a_{ij}(x)$ converging to
constants as $|x|\to\infty$.
Analogous Liouville type results can be derived in the parabolic case
using the same type of estimates (see e.~g.~\cite{Lady}).
The case $c\not\equiv0$ has been
treated in many papers, using different techniques, such as
probabilistic methods, semigroup or potential theory. With a
purely pde approach, it is proved by Brezis, Chipot and Xie
in the recent work \cite{BCX},
that the LP holds for \eq{s-a} in the following cases: $N\leq2$
and $c\leq0$; $N>2,\ c\leq0$ and $c(x)\leq-c_0(x)$ for $|x|$ large, with
either $c_0(x)=C|x|^{-\beta}$, $C>0,\ \beta>2$, or $c_0$
nonnegative nontrivial periodic function. General nonvariational
operators such as $L$ defined in Section \ref{sec:main}
are also considered in \cite{BCX} and the LP
is derived for the equation
\formulaI{L=0}
Lu=0\quad\text{in }\R^N,
\formulaF
when $c\leq0$ and $c(x)\leq-C|x|^{-2},\ \sum_{i=1}^N
b_i(x)x_i\leq C$, for some $C>0$ and $|x|$ large.
For the parabolic case, Hile and Mawata proved in \cite{HM} that the
LP holds for
a class of quasilinear equations satisfying some conditions at
infinity. Their result applies in particular to the equation \eq{P=0}
when $c\leq0$,
$(a_{ij}(X))_{i,j}\to\;$identity and $b_i(X),c(X)\to\;$constant, with a suitable
rate, as $|X|\to\infty$.

Some authors treated the problem of the existence and uniqueness of
nonnegative bounded solutions to linear
elliptic equations in divergence form from the point of view of
the criticality property of the operator. Starting from
the ideas of Agmon \cite{A} and Pinchover, and combining analytic and
probabilistic techniques (such as the Martin representation theorem)
Pinsky showed in \cite{P95}
that if $L$ is a periodic operator
satisfying $\lp=0$, then the unique (up to positive multiples)
positive bounded solution of \eq{L=0} is $\vp_p$ (see \cite{PiBook} for an
extensive treatment of the subject). This result is a particular case of
\thm{per} part (ii) above and, since when $c\equiv0$ there is no
difference between studying bounded solutions and positive bounded
solutions, it contains the case 1) of Corollary \ref{cor:liouville}
when one restricts it to elliptic equations
(except for the fact that some stronger regularity assumptions
are required in \cite{P95}).

Another related topic is that of the characterization of polynomial
growing solutions
of periodic equations in the whole space. We stress out that the
LP - as intended in the present paper - is obtained as a
particular case considering polynomials of degree zero. In that
framework, using some homogenization techniques, Avellaneda and
Lin \cite{AL}, and later Moser and Struwe \cite{MS}, proved that
the LP holds for \eq{s-a} if $c\equiv0$ and the $a_{ij}$ are
periodic (with the same period). The results of \cite{AL} and
\cite{MS} have been improved by Kuchment and Pinchover \cite{KP}
to the general non-self-adjoint elliptic equation \eq{L=0}, with $L$
periodic and $a_{ij},\ b,\ c$ smooth (see also Li and Wang
\cite{LW} for the case $a_{ij}$ measurable and $b_i,c\equiv0$).
The restriction
to bounded solutions of Theorem 28 part 3 in \cite{KP} is
equivalent to the restriction to stationary solutions of
statements (ii) and (iii) of \thm{per} here.
However, the method used in \cite{KP} - based on the Floquet theory -
is quite involving and it is not clear to us whether or not it
adapts to operators with non-smooth coefficients.
%

To our knowledge, no results about operators periodic in just one
variable, such as \thm{1per}, have been
previously obtained.

For nonlinear operators, the LP simply refer to uniqueness of
bounded (sometimes nonnegative) solutions in unbounded domains. The
following works - amongst many others - deal with this subject in the
elliptic case:
\cite{GiSp1}, \cite{BC-DN}, \cite{LZ} (semilinear
operators, see also \cite{B-V} for the parabolic case),
\cite{DD} (quasilinear operators), \cite{CC}, \cite{CL},
\cite{CDC}, \cite{R1} (fully nonlinear operators).\\

There is a vast literature on the problem of almost periodicity of
bounded solutions of linear equations with a.~p.~coefficients (see
e.~g.~\cite{amerio-prouse}, \cite{Fink}, \cite{Ph}, \cite{HM2}).
Usually ordinary differential equations or systems are considered,
often of the first order. As emphasized in \cite{apC-E}, some
authors made use in proofs of the claim that any bounded solution
in $\R$ of a second order linear elliptic equation with
a.~p.~coefficients has to be a.~p. This claim is false, as shown
by Counterexample \ref{c-e} and also by a counterexample in
\cite{apC-E}. There, the authors constructed an a.~p.~function
$c(x)$ such that the equation
$$u''+c(x)u=0\quad\text{in
}\R,$$ admits bounded solutions which are not a.~p. In their case,
the space of bounded solutions has dimension one and then the
LP holds. They also addressed the following open
question: if every solution of a linear equation in $\R$ with
a.~p.~coefficients is bounded are all solutions necessarily a.~p.?
Counterexample \ref{c-e} shows that the answer is no. A negative
answer was also given in \cite{HM1}, where the authors exhibite a class
of linear ordinary differential equations of order $n\geq2$ for
which all solutions are bounded in $\R$, yet no nontrivial
solution is a.~p. Thus, this also provides an example where the
LP does not hold, but it is not interesting in this
sense because the zero order term considered there is not nonpositive.


\subsection{Organization of the paper}

In Section \ref{sec:per}, we consider the case when $P$ and $f$ are
periodic and we prove \thm{1per}, Corollary \ref{cor:liouville} and \thm{per}.
In order to prove the periodicity of any bounded
solution $u$, we show that the difference between $u$ and
its translation by one period is identically equal to 0. This is acheived
by passing to a limit equation and making use of a supersolution
$v$ with positive infimum. We
take $v\equiv1$ in the case of \thm{1per} and $v\equiv\varphi_p$ in
the case of \thm{per}. We further derive
the existence and uniqueness of bounded entire solutions to \eq{P=f}
when $P=\partial_t-L$ and $L$ is periodic and satisfies $\lambda_p(-L)>0$.

Section \ref{sec:c-e} is devoted to the construction of the
function $b$ of Counterexample \ref{c-e}, which will be defined by
an explicit recursive formula.

\thm{ap} is proved in Section \ref{sec:ap}. The basic idea to
prove statement (i) is that, up to subsequences, all subsequences
of a given sequence of translations of $u$ converge to a solution
of the same equation. Also, one can come back to the original
equation by translating in the opposite direction. Then, the
result follows from \thm{per} part (iii).

In Sections \ref{sec:Omega}, we derive results analogous to
Theorems \ref{thm:1per} and \ref{thm:per} for the
Dirichlet and the Robin problems in periodic domains.
There, the periodic \pe\ $\lambda_p(-L)$ is replaced respectively by
$\lambda_{p,D}(-L)$ (seev Section \ref{sec:D}) and
$\lambda_{p,\mc{N}}(-L)$ (see Section \ref{sec:N})
which take into account the boundary conditions. Existence and
uniqueness results are presented as well.


\section{The LP for periodic operators}\label{sec:per}

Let us preliminarily reclaim the notion of periodic \pe\ and
eigenfunction. If $L$ is periodic then the Krein Rutman theory
yields the existence of a unique real number $\lambda$, called
periodic \pe\ of $-L$ (in $\R^N$), such that the eigenvalue
problem
$$\left\{\begin{array}{l}
-L\varphi=\lambda\varphi \quad \text{in }\R^N\\
\varphi\text{ is periodic, with the same period as }L
\end{array}\right.$$
admits positive solutions. Furthermore, the positive solution
$\varphi$ is unique up to a multiplicative constant, and
it is called periodic principal eigenfunction. We denote by
$\lambda_p(-L)$ and $\varphi_p$ respectively the periodic \pe\
and eigenfunction of $-L$.

The next lemma is the key tool to prove our results for periodic operators.

\begin{lemma}\label{lem:1per}
Assume that the operator $P$ and the function $f$
are periodic in the $m$-th variable, with the same period $l_m$.
If there exists a bounded function $v$ satisfying
$$\inf_{\R^{N+1}}v>0,\qquad Pv=\phi\;\text{ for }x\in\R^N,\ t\in\R,$$
for some nonnegative
function $\phi\in L^\infty(\R^{N+1})$, then any bounded solution $u$
of \eq{P=f} is periodic in the $m$-th variable, with period $l_m$.
\end{lemma}

\begin{proof}
Let $u$ be a bounded solution of \eq{P=f}. Define the functions
$$\psi(X):=u(X+l_me_m)-u(X),\qquad w(X):=\frac{\psi(X)}{v(X)}.$$ We
want to show that they are nonpositive.
Assume by way of contradiction that
$k:=\sup_{\R^{N+1}}w>0,$ and consider a sequence $\seq{X}$ in
$\R^{N+1}$ such that $w(X_n)\to k$. Define the sequence of functions
$\psi_n(X):=\psi(X+X_n)$. Since the $\psi_n$ are uniformly bounded and
satisfy
\formulaI{psin}
\partial_t \psi_n-a_{ij}(X+ X_n)\partial_{ij}\psi_n-
b_i(X+X_n)\partial_i\psi_n-c(X+X_n)\psi_n=0\quad\text{in }\R^{N+1},
\formulaF
interior parabolic estimates together
with the Rellich-Kondrachov compactness theorem imply that (a
subsequence of) the sequence $\seq{\psi}$ converges locally uniformly in
$\R^{N+1}$ to a bounded function $\psi_\infty$ and that
$\partial_t\psi_n\to\partial_t\psi_\infty ,\
\partial_i\psi_n \to\partial_i\psi_\infty,\
\partial_{ij}\psi_n \to\partial_{ij}\psi_\infty$ weakly in $L^p_{loc}(\R^{N+1})$,
for any $p>1$.
Let $\t a_{ij},\t
b_i$ be the locally uniform limits and $\t c$ be the weak limit in
$L^p_{loc}(\R^{N+1})$ of a
converging subsequence respectively of $a_{ij}(X+ X_n),\
b_i(X+X_n)$ and $c(X+X_n)$. Thus, passing to the weak limit in
\eq{psin} we derive
$$\t P\psi_\infty=0\quad\text{in }\R^{N+1},$$
where
$$\t P:=\partial_t-\t a_{ij}(X)\partial_{ij}-\t
b_i(X)\partial_i -\t c(X).$$
Analogously, the functions $v(X+X_n)$ converge (up to subsequences)
locally uniformly in $\R^{N+1}$ to a function $v_\infty$ satisfying
$$\inf_{\R^{N+1}}v_\infty>0,\qquad\t P v_\infty\geq0\quad\text{in }\R^{N+1}.$$
The function $w_\infty:=\psi_\infty/v_\infty$ reaches its maximum
value $k$ at $0$. Moreover,
$$0=\frac{\t P\psi_\infty}{v_\infty}=\partial_t w_\infty-\t M w_\infty
+\frac{\t Pv_\infty}{v_\infty}w_\infty\quad\text{in }\R^{N+1},$$ where
the operator $\t M$ is defined by
$$\t M:=\t a_{ij}\partial_{ij}+(2v_\infty^{-1}\t a_{ij}\partial_j
{v_\infty}+\t b_i)\partial_i.$$ Since the term $(\t
Pv_\infty)/v_\infty$ is nonnegative, we can apply the parabolic
strong maximum principle to the function $w_\infty$ (see
\cite{Max} for the smooth case and \cite{Lady}, \cite{Lie} for the
case of strong solutions) and derive
$$\fa x\in\R^N,\ t\leq0,\quad
w_\infty(x,t)=k.$$ Using a diagonal method, we can find a
subsequence of $\seq{X}$ (that we still call $\seq{X}$) and a sequence
$(\zeta_h)_{h\in\N}$ in $[0,\sup u]$ such that
$$\fa h\in\N,\quad \limn u(-hl_me_m+X_n)=\zeta_h.$$
As a consequence,
$$\fa h\in\N,\quad
\zeta_h-\zeta_{h+1}=\psi_\infty(-(h+1)l_me_m)=kv_\infty(-(h+1)l_me_m),$$
and then $\lim_{h\to\infty}\zeta_h=-\infty$: contradiction.
We have shown that $w\leq0$, that is, $u(X+l_me_m)\leq u(X)$ for
$X\in\R^N$. The opposite
inequality can be obtained following the same arguments, with
$l_m$ with $-l_m$. This time, the contradiction reached is that
the sequence $(\zeta_h)_{h\in\N}$ as defined above goes to $+\infty$
as $h\to\infty$.
\end{proof}

\begin{proofof}{\thm{1per}}
Apply Lemma \ref{lem:1per} with $v\equiv1$.
\end{proofof}

\begin{proofof}{Corollary \ref{cor:liouville}}
If $u$ is a bounded solution to \eq{P=0} then \thm{1per}
implies that $u$ is periodic (in all the variables). In particular, it attains
its maximum $M$ and minimum $m$ in $\R^{N+1}$ at some points $(x_M,t_M)$ and
$(x_m,t_m)$ respectively. Hence, the \SMP\ implies that if $u_M\geq0$
then $u(x,t)=M$ for $t\leq t_M$
and $x\in\R^N$, otherwise $u(x,t)=m$ for $t\leq t_m$
and $x\in\R^N$. Therefore, $u$ is constant because it is periodic in $t$.
The statement then follows.
\end{proofof}

\begin{proofof}{\thm{per}}
(i) The function $v(x,t):=\varphi_p(x)$ is bounded and satisfies
$$\inf_{\R^{N+1}}v>0,\qquad Pv=\phi\ \text{ for }x\in\R^N,\ t\in\R,$$
with $\phi=\lp\varphi_p\geq0$. Hence, the statement is a
consequence of Lemma \ref{lem:1per}.

(ii) Up to replace $u$ with $-u$, it is not restrictive to assume
that $f\leq0$. Set
$$k:=\sup_{x\in\R^N\atop t\in\R} \frac{u(x,t)}{\varphi_p(x)}.$$
Since $u$ is periodic by (i) - with the same space period $(l_1\pp l_N)$
as $\varphi_p$ - it follows that there exists
$X_0\in[0,l_1)\times\cdots\times[0,l_{N+1})$ where
the nonnegative
function $w(x,t):=k\varphi_p(x)-u(x,t)$ vanishes. Furthermore,
$$Pw=k\lambda_p(-L)\varphi_p-f\geq0\ \text{ in }\R^{N+1}.$$
Therefore, the strong \MP\ and the time-periodicty of $w$ yield
$w\equiv0$, that is, $u\equiv k\varphi_p$ and $f\equiv0$.

(iii) Suppose that $\sup u\geq0$ (otherwise replace $u$ with $-u$).
Proceeding as in (ii), one can find a constant $k\geq0$ such that
the periodic function $w(x,t):=k\varphi_p(x)-u(x,t)$ is nonnegative, vanishes
at some point $X_0\in\R^{N+1}$ and satisfies
$Pw=k\lambda_p(-L)\varphi_p\geq0$. Once again, the strong \MP\
implies $w\equiv0$. Hence, $k\lambda_p(-L)\varphi_p\equiv0$, that is,
$k=0$ and then $u\equiv0$.
\end{proofof}


\begin{remark}\label{rem:c}\rm
If $L$ is periodic and $c\equiv0$ then $\lambda_p(-L)=0$, with
$\varphi_p\equiv1$. If $c\leq0$, $c\not\equiv0$, then
$\lambda_p(-L)>0$, as it is easily seen by applying the strong
\MP\ to $\varphi_p$. Hence, in the case $P=\partial_t-L$,
Corollary \ref{cor:liouville} is contained in
\thm{per} parts (ii) and (iii). Furthermore, the
existence and uniqueness result of Corollary \ref{cor:!} below apply when
$c\leq0$, $c\not\equiv0$.
\end{remark}

By \thm{per} part (iii), $\lambda_p(-L)>0$ implies the uniqueness
of bounded entire solutions of $\partial_t u-Lu=f$. Indeed, it is
also a sufficient condition for the existence when $f$ is bounded.

\begin{corollary}\label{cor:!}
If $P=\partial_t-L$, with $L$ periodic such that $\lambda_p(-L)>0$,
and $f\in L^\infty(\R^{N+1})$ then \eq{P=f} admits a unique
bounded solution.
\end{corollary}

\begin{proof}
A standard method to construct entire solutions in the whole space is to
consider the limit as $r\to\infty$ of solutions $u_r$ of (for
instance) the Dirichlet problems
\formulaI{P=fr}
\left\{\begin{array}{ll}
Pu_r=f(x,t), & x\in B_r,\ t\in(-r,r)\\
u_r=0, & x\in\partial B_r,\ t\in(-r,r)\\
u_r=0, & x\in B_r,\ t=-r
\end{array}\right.
\formulaF
(here, $B_r$ denotes the ball in $\R^N$ with radius $r$ and centre
$0$).
The so obtained solution
is bounded provided that the family $(u_r)_{r>0}$ is uniformly
bounded. This will
follow from the strict positivity of $\lambda_p(-L)$.
Define the function
$$v(x):=\frac{\|f\|_{L^\infty(\R^{N+1})}}{\lambda_p(-L)\min_{\R^N}
\vp_p}\vp_p(x).$$
Since $-v$ and $v$ are respectively a sub and a supersolution of
\eq{P=fr}, the parabolic comparison principle yields
$$\fa r>0,\quad-v\leq u_r\leq v\ \text{ in }B_r\times(-r,r).$$
Thus, using interior
estimates and the embedding theorem, we can find a diverging sequence
$(r_n)_{n\in\N}$ such that $(u_{r_n})_{n\in\N}$ converges locally
uniformly in $\R^{N+1}$ to a bounded solution of \eq{P=f}.
The uniqueness result is an immediate consequence of \thm{per} part (iii).
\end{proof}


We remark that if $f$ is periodic then one can prove the existence
result of Corollary \ref{cor:!} by a standard functional method:
after regularizing the operator in order to write it in divergence
form, one considers the problem in the space of periodic functions
and, owing to \thm{per} part (iii), applies the Fredholm
alternative to the inverse operator. Also, note that by \thm{per}
part (ii), the equation $\partial_t u-\partial_{xx}u=1$ does not
admit entire bounded solutions and then the hypothesis
$\lambda_p(-L)>0$ is sharp for the existence result of Corollary
\ref{cor:!}.


\section{Counter-example when $L$ is almost periodic}\label{sec:c-e}

This section is devoted to the construction of Counterexample \ref{c-e}.
Note that, by the uniqueness of solutions of the Cauchy problem,
any non-constant solution of \eq{c-e} must be strictly monotone.

We first construct a discontinuous function $\sigma$, then we
modify it to obtain a Lipschitz continuous limit periodic function
$b$. Let us reclaim the definition of limit periodic functions,
which are a proper subset of a.~p.~functions.

\begin{definition}\label{def:lp}
We say that a function $\phi\in C(\R^N)$ is limit periodic if
there exists a sequence of continuous periodic functions
converging uniformly to $\phi$ in $\R^N$.
\end{definition}

We start defining $\sigma$ on the interval $(-1,1]$:
$$\sigma(x)=\left\{\begin{array}{ll}
-1&\text{ if }-1<x\leq0,\\
1&\text{ if }\;0<x\leq1.
\end{array}\right.$$
Then in $(-3,3]$ setting
$$\fa x\in(-3,-1],\qquad\sigma(x)=\sigma(x+2)-1,$$
$$\fa x\in(1,3],\qquad\sigma(x)=\sigma(x-2)+1,$$
and, by iteration,
\formulaI{iter-}\fa
x\in(-3^{n+1},-3^n],\qquad\sigma(x)=\sigma(x+2\.3^n)-\frac1{(n+1)^2},
\formulaF
\formulaI{iter+}\fa
x\in(3^n,3^{n+1}],\qquad\sigma(x)=\sigma(x-2\.3^n)+\frac1{(n+1)^2}.
\formulaF
By construction, the function $\sigma$ satisfies
$\norma\sigma=1+\sum_{n=1}^\infty n^{-2}$, and it is odd except
for the set $\Z$, in the sense that $\sigma(-x)=-\sigma(x)$ for
$x\in\R\meno\Z$.

\begin{proposition}\label{pro:lp}
There exists a sequence of bounded periodic functions
$(\phi_n)_{n\in\N}$ converging uniformly to $\sigma$ in $\R$ and
such that
$$\fa n\in\N,\qquad\phi_n\in C(\R\meno\Z),\qquad\phi_n\text{ has
  period }2\. 3^n.$$
\end{proposition}

\begin{proof}
Fix $n\in\N$. For $x\in(-3^n,3^n]$ set $\phi_n(x):=\sigma(x)$,
then extend $\phi_n$ to the whole real line by periodicity, with
period $2\. 3^n$. We claim that
$$\norma{\sigma-\phi_n}\leq\sum_{k=n+1}^\infty\frac1{k^2},$$
which would conclude the proof. We prove our claim by a recursive
argument, showing that the property
$$(\mathcal{P}_i)\qquad\fa
x\in(-3^{n+i},3^{n+i}],\qquad |\sigma(x)-\phi_n(x)|\leq
\sum_{k=n+1}^{n+i}\frac1{k^2}$$ holds for every $i\in\N$. Let us
check $(\mathcal{P}_1)$. By \eq{iter-} and \eq{iter+} we get
$$\sigma(x)=\left\{\begin{array}{ll}
\displaystyle \sigma(x+2\.3^n)-\frac1{(n+1)^2}&
\quad\text{if }-3^{n+1}<x\leq-3^n\\
\phi_n(x) & \quad\text{if } -3^n< x\leq 3^n\\
\displaystyle \sigma(x-2\.3^n)+\frac1{(n+1)^2}&
\quad\text{if }\; 3^n<x\leq 3^{n+1}\ .\\
\end{array}\right.$$
Property $(\mathcal{P}_1)$ then follows from the periodicity of $\phi_n$.

Assume now that $(\mathcal{P}_i)$ holds for some $i\in\N$.
Let $x\in(-3^{n+i+1},3^{n+i+1}]$. If $x\in(-3^{n+i},3^{n+i}]$
then $$|\sigma(x)-\phi_n(x)|\leq\sum_{k=n+1}^{n+i}\frac1{k^2}\leq
\sum_{k=n+1}^{n+i+1}\frac1{k^2}.$$
Otherwise, set
$$y:=\left\{\begin{array}{ll}
x+2\.3^{n+i} & \text{if }x<0\\
x-2\.3^{n+i} & \text{if }x>0\ .\\
\end{array}\right.$$
Note that $y\in(-3^{n+i},3^{n+i}]$ and $|x-y|=2\.3^{n+i}$.
Thus, \eq{iter-}, \eq{iter+}, $(\mathcal{P}_i)$ and the periodicity
of $\phi_n$ yield
\begin{equation*}\begin{split}
|\sigma(x)-\phi_n(x)| &\leq|\sigma(x)-\sigma(y)|+|\sigma(y)-\phi_n(y)|\\
&\leq \frac1{(n+i+1)^2}+\sum_{k=n+1}^{n+i}\frac1{k^2}\\
&= \sum_{k=n+1}^{n+i+1}\frac1{k^2}.
\end{split}\end{equation*}
This means that $(\mathcal{P}_{i+1})$ holds and then the proof is concluded.
\end{proof}

Note that $\sigma$ is not limit periodic because it is discontinuous
on $\Z$.

\begin{proposition}
The function $\sigma$ satisfies \formulaI{intf} \fa
x\geq1,\qquad\int_0^x\sigma(t)dt\geq\frac x{2(\log_3x+1)^2}.
\formulaF
\end{proposition}

\begin{proof}
For $y\in\R$, define $F(y):=\int_0^y\sigma(t)dt$. Let us
preliminarily show that, for every $n\in\N$, the following formula
holds: \formulaI{intn} \fa y\in[0,3^n],\qquad
F(y)\geq\frac y{2n^2}. \formulaF We shall do it by
iteration on $n$. It is immediately seen that \eq{intn} holds for
$n=1$. Assume that \eq{intn} holds for some
$n\in\N$. We want to prove that \eq{intn} holds with $n$ replaced by
$n+1$. If $y\in[0,3^n]$ then
$$F(y)\geq\frac y{2n^2}\geq\frac y{2(n+1)^2}.$$
If $y\in(3^n,2\.3^n]$ then, by computation,
$$F(y)=F(2\.3^n-y)+\int_{2\.3^n-y}^y\sigma(t)dt\geq\frac{2\.3^n-y}{2n^2}+
\int_{-(y-3^n)}^{y-3^n}\sigma(\tau+3^n)d\tau.$$ Using property \eq{iter+},
one sees that
\begin{equation*}\begin{split}
\int_{-(y-3^n)}^{y-3^n}\sigma(\tau+3^n)d\tau &=
\int_{-(y-3^n)}^0\sigma(\tau+3^n)d\tau
+\int_0^{y-3^n}\sigma(\tau-3^n)d\tau+\frac{y-3^n}{(n+1)^2}\\
&= \frac{y-3^n}{(n+1)^2}\;,
\end{split}\end{equation*}
where the last equality holds because $\sigma$ is odd except
in the set $\Z$. Hence,
$$F(y)\geq
\frac{2\.3^n-y}{2n^2}+\frac{y-3^n}{(n+1)^2}\geq\frac
y{2(n+1)^2}.$$
Let now $y\in(2\.3^n,3^{n+1}]$.
Since $F(2\.3^n)\geq 3^n(n+1)^{-2}$, as we
have seen before, and \eq{iter+} holds, it follows that
$$F(y)=F(2\.3^n)+\int_{2\.3^n}^y\sigma(t)dt\geq
\frac{3^n}{(n+1)^2}+F(y-2\.3^n)+\frac{y-2\.3^n}{n+1)^2}.$$
Using the hypothesis \eq{intn} we then get
$$F(y)\geq\frac{y-3^n}{(n+1)^2}+\frac{y-2\.3^n}{2n^2}\geq\frac
y{2(n+1)^2}.$$

We have proved that \eq{intn} holds for any $n\in\N$.
Consider now $x\geq1$. We can find an integer $n=n(x)$ such that
$x\in[3^{n-1},3^n)$.
Applying \eq{intn} we get $F(x)\geq x(2n^2)^{-1}$. Therefore,
since $n\leq\log_3x+1$, we infer that
$$F(x)\geq\frac x{2(\log_3x+1)^2}.$$
\end{proof}

In order to define the function $b$, we introduce the following auxiliary
function $z\in C(\R)$ vanishing on $\Z$: $z(x):=2|x|$ if
$x\in[-1/2,1/2]$, and it is extended by periodicity with period $1$
outside $[-1/2,1/2]$.
Then we set
$$b(x):=\sigma(x)z(x).$$
The definition of $b$ is easier to understand by its graph (see
Figure \ref{fb}).
%
%
%
\psfrag{sigma}{$=\sigma(x)$} \psfrag{b}{$=b(x)$}
\immpiccola{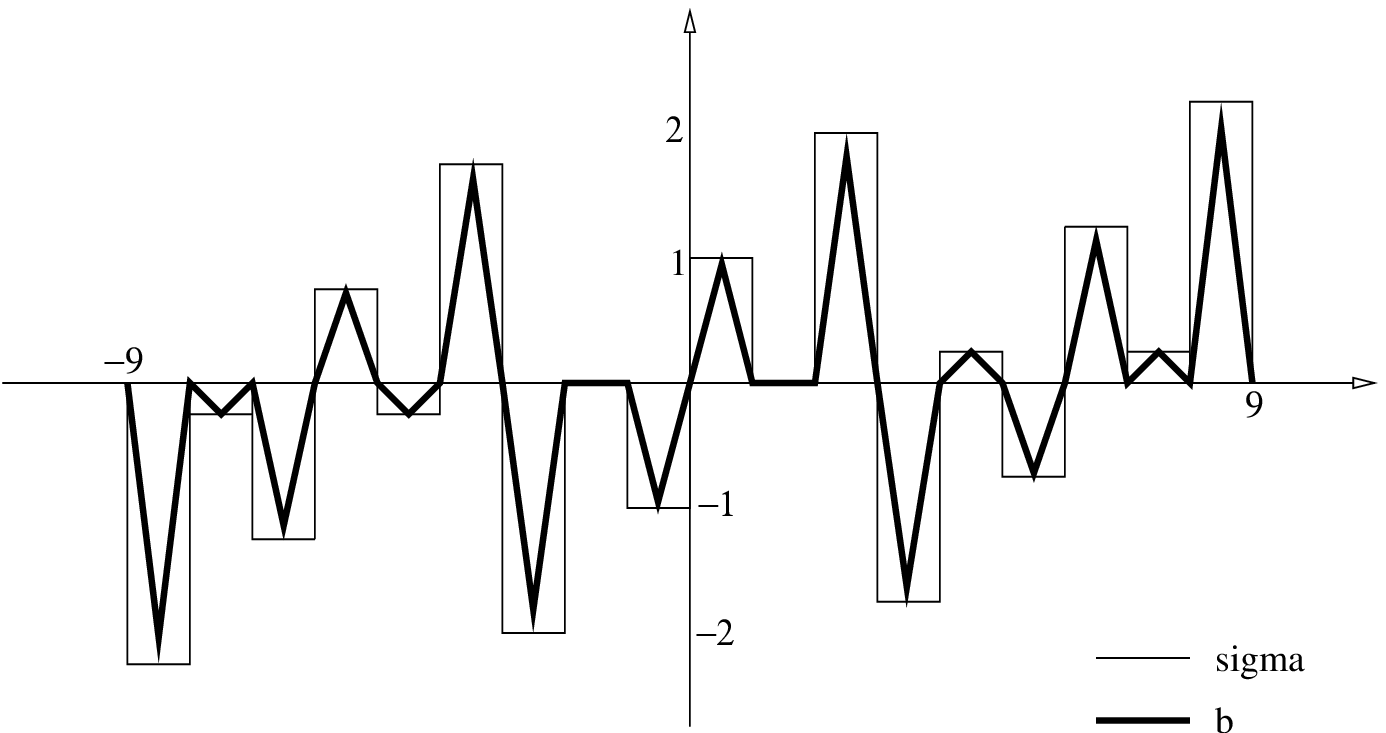}{}{12}{graphs of $\sigma$ and $b$}{fb}

\begin{proposition}
The function $b$ is odd and limit periodic.
\end{proposition}

\begin{proof}
Let us check that $b$ is odd. For $x\in\Z$ we find
$b(-x)=0=-b(x)$, while, for $x\in\R\meno\Z$,
$$b(-x)=\sigma(-x)z(-x)=-\sigma(x)z(x)=-b(x).$$
In order to prove that $b$ is limit periodic, consider the sequence of
periodic functions
$(\phi_n)_{n\in\N}$ given by Proposition
\ref{pro:lp}. Then define
$$\psi_n(x):=\phi_n(x)z(x).$$
Clearly, the functions $\psi_n$ are continuous (because $z$ vanishes on $\Z$)
and periodic, with period
$2\.3^n$ (because $z$ has period $1$). Also, for $n\in\N$,
$$|b-\psi_n|=|\sigma-\phi_n|z\leq|\sigma-\phi_n|.$$
Therefore, $\psi_n$ converges uniformly to $b$ as $n$ goes to
infinity.
\end{proof}

\begin{proposition}\label{pro:c-e}
All solutions of \eq{c-e} are bounded and they are generated
by $u_1\equiv1$ and a non-a.~p.~function $u_2$.
\end{proposition}

\begin{proof}
The two-dimensional space of solutions of \eq{c-e} is generated by
$u_1\equiv1$ and
$$u_2(x):=\int_0^x\exp\left(-\int_0^yb(t)dt\right)dy.$$
Since $u_2$ is strictly increasing, it cannot be a.~p. So, to
prove the statement it only remains to show that $u_2$ is bounded.
By construction, it is clear that, for $m\in\Z$,
$$\int_0^mb(t)dt=\frac12\int_0^m\sigma(t)dt.$$
Consequently, by \eq{intf}, we get for $x\geq 1$
$$\int_0^x b(t)dt=\frac12\int_0^{[x]}
\sigma(t)dt+\int_{[x]}^xb(t)dt\geq\frac{x-1}{4(\log_3x+1)^2}-\norma{b}$$
and then \begin{equation*}\begin{split} 0\leq u_2(x) &\leq
e^{\norma{b}}
\int_0^x\exp\left(-\frac{y-1}{4(\log_3y+1)^2}\right)dy\\
& \leq e^{\norma{b}}
\int_0^{+\infty}\exp\left(-\frac{y-1}{4(\log_3y+1)^2}\right)dy.
\end{split}\end{equation*}
Since $b$ is odd, it follows that $u_2$ is odd too and then it is
bounded on $\R$.
\end{proof}

\begin{remark}\rm
The function $b=\sigma z$ we have constructed before is uniformly
Lipschitz continuous, with Lipschitz constant equal to
$2\|\sigma\|_{L^\infty(\R)}$.  Actually, one could use a suitable
$C^\infty$ function instead of $z$ in order to obtain a function
$b\in C^\infty(\R)$.
\end{remark}

\begin{remark}\rm
The reason why the LP fails to hold in the
a.~p.~case is that, as shown by the previous counterexample, an
a.~p.~linear equation with nonpositive zero order coefficient may
admit non-a.~p.~bounded solutions in the whole space. Instead, the
space of a.~p.~solutions of \eq{L=0}, with $c\leq0$ and without
any almost periodicity assumptions on $L$, has at most dimension
one, that is, the LP holds if all bounded solutions
are a.~p. More precisely, the result of Corollary
\ref{cor:liouville} holds true if one requires $u$ to be a.~p.,
even by dropping the periodicity assumption on $L$. To see this,
consider an a.~p.~solution $u$ of \eq{L=0}. Up to replace $u$ with
$-u$, we can assume that $U:=\sup u\geq0$. Let $\seq{x}$ be a
sequence in $\R^N$ such that $u(x_n)\to U$. Then, up to
subsequences, the functions $u_n(x):=u(x+x_n)$ converge locally
uniformly in $x\in\R^N$ to a solution $u_\infty$ of a linear
equation $-\t L=0$ in $\R^N$, with nonpositive zero order term
(see the arguments in the proof of Lemma \ref{lem:1per}). The
\SMP\ then yields $u_\infty\equiv U$. Since the convergence of a
subsequence of $u_n$ is also uniform in $\R^N$, by the almost
periodicity of $u$, we find that $u\equiv U$. Therefore, the
conclusion of Corollary \ref{cor:liouville} holds.
\end{remark}


\section{Sufficient conditions for the almost periodicity of solutions}
\label{sec:ap}

\begin{proofof}{\thm{ap}}
(i) consider an arbitrary sequence $\seq{X}=((x_n,t_n))_{n\in\N}$ in
$\R^N\times\R$. Since
$a_{ij},\ b_i,\ c$ and $f$ are a.~p.~(because $a_{ij},b_i,c\in
C(\R^N)$ are periodic)
there exists a subsequence of $\seq{X}$
(that we still call $\seq{X}$) such that $a_{ij}(x+x_n),\
b_i(x+x_n),\ c(x+x_n)$ and $f(x+x_n,t+t_n)$ converge uniformly in
$x\in\R^N,\ t\in\R$. We claim that $u(X+X_n)$ converges uniformly in
$X\in\R^{N+1}$ too. Assume by contradiction that this is not the
case. Then, there exist $\e>0$, a sequence
$\seq{Y}=((y_n,\tau_n))_{n\in\N}$ in $\R^N\times\R$ and
two subsequences $\seq{X^1}$ and $\seq{X^2}$ of $\seq{X}$ such
that \formulaI{assurdo}\fa n\in\N,\qquad
|u(Y_n+X_n^1)-u(Y_n+X_n^2)|>\e. \formulaF For $\sigma=1,2$ set
$\seq{Z^\sigma}:=(Y_n+X_n^\sigma)_{n\in\N}$. Applying again the definition
of almost periodicity, we can find a common sequence
$(n_k)_{k\in\N}$ in $\N$ such that, for $\sigma=1,2$, the functions
$f(X+Z^\sigma_{n_k})$ converge to some functions
$f^\sigma$ uniformly in
$X\in\R^{N+1}$. As
$f(X+X_n)$ converges uniformly in $X\in\R^{N+1}$, we see that
$$\fa x\in\R^{N+1},\quad
f^1(X)=\lim_{k\to\infty}f(X+Y_{n_k}+X_{n_k}^1)=
\lim_{k\to\infty}f(X+Y_{n_k}+X_{n_k}^2)=f^2(X).$$
Let $(\eta_k)_{k\in\N}$ be a sequence in
$[0,l_1)\times\cdots\times[0,l_N)$ such that
$y_{n_k}+x^1_{n_k}-\eta_k\in\prod_{i=1}^N l_i\Z$ and let $\eta$ be the limit of
(a subsequence of) $(\eta_k)_{k\in\N}$. Owing to the periodicity and the
uniform continuity of $c$, we get
$$c(x+\eta)=\lim_{k\to\infty}c(x+\eta_k)=
\lim_{k\to\infty}c(x+y_{n_k}+x^1_{n_k})=
\lim_{k\to\infty}c(x+y_{n_k}+x^2_{n_k}),$$
uniformly in $x\in\R^N$. Analogously, for $\sigma=1,2$,
$$\lim_{k\to\infty}a_{ij}(x+y_{n_k}+x^\sigma_{n_k})=a_{ij}(x+\eta),\qquad
b_i(x+y_{n_k}+x^\sigma_{n_k})=b_i(x+\eta),$$
uniformly with respect to $x\in\R^N$. By standard
parabolic estimates and compact injection theorem, it follows that
there exists a subsequence of $(n_k)_{k\in\N}$ (that we still call
$(n_k)_{k\in\N}$) such that, for $\sigma=1,2$, the sequences
$u^\sigma_k:=u(\.+Z^\sigma_{n_k})$ converge locally uniformly in $\R^{N+1}$ to
some functions $u^\sigma$ and $\partial_t u^\sigma_k\to\partial_t
u^\sigma,\ \partial_i u^\sigma_k\to\partial_i u^\sigma
,\ \partial_{ij}u^\sigma_k\to\partial_{ij}u^\sigma$
weakly in $L^p_{loc}(\R^{N+1})$, for $p>1$. Passing to the weak limit
in the equations satisfied by the $u^\sigma_k$, we infer that
the $u^\sigma$ satisfy
\formulaI{uj}
\inf_{\R^{N+1}}u\leq u^\sigma\leq\sup_{\R^{N+1}}u, \qquad
\partial_t u^\sigma-L_\eta u^\sigma=f^1\quad\text{in }\R^{N+1},
\formulaF
where
$$L_\eta:=a_{ij}(\.+\eta)\partial_{ij}+b_i(\.+\eta)\partial_i+c(\.+\eta).$$
Clearly, if $\varphi_p$
is the periodic principal eigenfunction of $-L$, then
$\varphi_p(\.+\eta)$ is the periodic principal eigenfunction of $-L_\eta$.
This shows that $\lambda_p(-L_\eta)=\lambda_p(-L)\geq0$. As
$\partial_t(u^1-u^2)-L_\eta(u^1-u^2)=0$ in $\R^{N+1}$, statements (ii)
and (iii) of
\thm{per} yields
\formulaI{u12}
\fa x\in\R^N,\ t\in\R,\quad
u^1(x,t)-u^2(x,t)=k\varphi_p(x+\eta),
\formulaF
for some $k\in\R$. In order to show that
$k=0$, we come back to the original equation. For $\sigma=1,2$,
the following limits hold uniformly in $X=(x,y)\in\R^N\times\R$:
$$\lim_{k\to\infty}f^1(X-Z_{n_k}^\sigma)=
\lim_{k\to\infty}f((X-Z_{n_k}^\sigma)+Z_{n_k}^\sigma)=f(X),$$
$$\lim_{k\to\infty}a_{ij}(x+\eta-y_{n_k}-x_{n_k}^\sigma)=a_{ij}(x),\qquad
\lim_{k\to\infty}b_i(x+\eta-y_{n_k}-x_{n_k}^\sigma)=b_i(x),$$
$$\lim_{k\to\infty}c(x+\eta-y_{n_k}-x_{n_k}^\sigma)=c(x).$$
Therefore, with usual arguments, we see that,
for $\sigma=1,2$, $u^\sigma(\.-Z_{n_k}^\sigma)$ converges (up to subsequences)
locally uniformly to a function $v^\sigma$ satisfying
\formulaI{vj}
\inf_{\R^{N+1}}u^\sigma\leq v^\sigma\leq\sup_{\R^{N+1}}u^\sigma,
\qquad Pv^\sigma=f\quad\text{in }\R^{N+1}.
\formulaF
Hence, $P(u-v^\sigma)=0$ and
then, again by \thm{per} parts (ii)-(iii), there exists a constant
$h^\sigma\in\R$ such that $u-v^\sigma\equiv h^\sigma\varphi_p$. Since
$\inf_{\R^{N+1}}u\leq v^\sigma\leq\sup_{\R^{N+1}}u$ by \eq{uj} and \eq{vj}, we
infer that $h^1=h^2=0$, that is, $v^1\equiv v^2\equiv u$. Consequently,
$$\inf_{\R^{N+1}}u^1=\inf_{\R^{N+1}}u^2=\inf_{\R^{N+1}}u,\qquad
\sup_{\R^{N+1}}u^1=\sup_{\R^{N+1}}u^2=\sup_{\R^{N+1}}u,$$ and then, by
\eq{u12}, $u_1\equiv u_2$. This is a contradiction because, by
\eq{assurdo}, $|u^1(0)-u^2(0)|\geq\e$.

(ii) Up to replace $u$ with $-u$, we can assume that $f\leq0$. Set
$$k:=\sup_{x\in\R^N\atop t\in\R}\frac{u(x,t)}{\varphi_p(x)}$$
and $v(x,t):=k\varphi_p(x)-u(x,t)$. Thus, $v\geq0$ and there
exists a sequence $\seq{X}=((x_n,t_n))_{n\in\N}$ in $\R^N\times\R$
such that $\limn v(X_n)=0$.
Arguing as above, we find that (up to subsequences) $v(\.+X_n)$
converges localy uniformly to a
nonnegative function $\t v$ satisfying
$$\t v(0)=0,\qquad
\partial_t\t v-a_{ij}(\.+\eta)\partial_{ij}\t v-b_i(\.+\eta)\partial_i\t v-
c(\.+\eta)\t v\geq0 \quad\text{in }\R^{N+1},$$ for some
$\eta\in[0,l_1)\times\cdots\times[0,l_N)$. Applying the strong \MP,
we get $\t v(x,t)=0$ for $x\in\R^N,\ t\leq0$. As $v$ is a.~p.~by
part (i), we infer that $\limn v(X+X_n)=\t v(X)$ uniformly with respect to
$X\in\R^{N+1}$. Thus, $\lim_{t\to-\infty}v(x,t)=0$ uniformly in $x\in\R^N$.
Again by the almost periodicity, we can find a sequence $\seq{t}$ in
$\R$ tending to $-\infty$ and such that $v(x,t+t_n)$ converges
uniformly with respect to $(x,t)\in\R^N\times\R$. Since
$$\fa x\in\R^N,\ t\in\R,\quad\limn v(x,t+t_n)=0,$$
we derive $v\equiv0$.
\end{proofof}

Corollary \ref{cor:!} and \thm{ap} part (i) imply the existence
of a unique a.~p.~solution of \eq{P=f} when $P=\partial_t-L$, $L$ is periodic,
$\lambda_p(-L)>0$ and $f$ is a.~p.

We conclude this section with a result concerning solutions of
\eq{P=f} when $P$ is periodic and $f$ is uniformly continuous ($UC$)
and a.~p.~in just one variable, i.~e.~there exists
$m\in\{1,\cdots,N+1\}$ such that, for any $(X_1\pp
X_{m-1},X_{m+1}\pp X_{N+1})\in\R^{N+1}$, $X_m\mapsto f(X_1\pp X_m\pp
X_{N+1})$ is a.~p.

\begin{theorem}\label{thm:ap1}
Let $u$ be a bounded solution of \eq{P=f}, with $f\in UC(\R^{N+1})$
a.~p.~in the $m$-th variable and either $P$ periodic in the
$m$-th variable, $c\leq0$, or $P=\partial_t-L$, $L$ periodic,
$\lambda_p(-L)\geq0$. Then, $u$ is a.~p.~in the $m$-th variable.
\nota{check}
\end{theorem}

Owing to the next consideration, the proof of \thm{ap1} is
essentially the same as that of \thm{ap} part (i).

\begin{lemma}\label{lem:AA}
Let $\phi\in UC(\R^{N+1})$ be a.~p.~in the $m$-th variable. Then,
from any real sequence $\seq{s}$ can be extracted a subsequence
$(s_{n_k})_{k\in\N}$ such that, for all $(X_1\pp
X_{m-1},X_{m+1}\pp X_{N+1})\in\R^{N+1}$, the sequence $(\phi(X_1,$
$\cdots, X_m+s_{n_k}\pp X_{N+1}))_{k\in\N}$ converges uniformly in
$X_m\in\R$.
\end{lemma}

\begin{proof}
The proof is similar to that of the Arzela-Ascoli theorem. For
simplicity, consider the case $m=N+1$. Let $\seq{s}$ be a sequence
in $\R$. As for any $q\in\Q^N$ there exists a subsequence
$(s^q_{n})_{n\in\N}$ of $\seq{s}$ such that
$(\phi(q,t+s^q_n))_{n\in\N}$ converges uniformly in $t\in\R$,
using a diagonal method we can find a common subsequence
$(s_{n_k})_{k\in\N}$ such that $(\phi(q,t+s_{n_k}))_{k\in\N}$
converges uniformly in $t\in\R$, for all $q\in\Q^N$. Fix
$x\in\R^N$. By the uniform continuity of $\phi$, for any $\e>0$
there exists $q\in\Q^N$ such that
$$\fa t\in\R,\quad|\phi(x,t)-\phi(q,t)|<\frac\e3.$$
Therefore,
$$|\phi(x,t+s_{n_k})-\phi(x,t+s_{n_h})|<\frac23\e+
|\phi(q,t+s_{n_k})-\phi(q,t+s_{n_h})|<\e$$ for $h,k$ big enough,
independent of $t\in\R$.
\end{proof}

\begin{remark}\rm
Statement (i) of \thm{ap} does not follow from \thm{ap1} because
being a.~p.~separately in each variable does not imply the almost
periodicity in the sense of Definition \ref{def:ap}. For example,
the function $\phi(x,y)=\sin(xy)$ is periodic in each variable but
it is not a.~p., because it is known that any a.~p.~function is
uniformly continuous (see e.~g.~\cite{amerio-prouse})
\end{remark}


\section{General periodic domains}\label{sec:Omega}

Henceforth, $\O$ denotes a {\bf uniformly smooth}
domain in $\R^N$. The symbol $\nu$ stands for the outer unit
normal vector field to $\O$.

We fix $l_1\pp l_{N+1}>0$. The domain $\O\subset\R^N$ is said to
be periodic in the direction $x_m$, $m\in\{1,\pp N\}$, if
$\O+\{l_me_m\}=\O$. If $\O$ is periodic in all directions, we
simply say that it is periodic. From now on, when we say that a
function or an operator is periodic (resp.~periodic in the $m$-th
variable with $m\in\{1\pp N+1\}$) we mean that its period is
$(l_1\pp l_{N+1})$ (resp.~$l_m$).

Besides the assumptions of Section \ref{sec:main}, we will
sometimes require in the sequel that the coefficients of $P$ and
the function $f$ are uniformly H\"older continuous
\footnote{ we denote by $C^{2n+\gamma,n+\frac\gamma2}$, with $n\in\{0,1\}$
and $\gamma\in[0,1)$, the space of functions 
whose space derivatives up to order
$2n$ and time derivative, if $n=1$, are {\bf locally} H\"older
continuous with exponent $\gamma$ with respect to $x$ and with exponent
$\gamma/2$ with respect to $t$.\\ 
If these derivatives are {\bf uniformly} H\"older
continuous then we say that the function belongs to $C^{2n+\gamma,n+\frac\gamma2}_b$.}.
This is
because, in some proofs, we need the solutions to be
Lipschitz continuous. In the elliptic case, this property follows from 
$W^{2,p}$ estimates, for $p>N$, and embedding theorem and indeed
the H\"older continuity assumption is not necessary.


\subsection{Dirichlet boundary conditions}\label{sec:D}

\def\pD{\varphi_{p,D}}

We deal with the Dirichlet
problem \formulaI{PD}
\left\{\begin{array}{ll}
Pu=f(x,t), & x\in\O,\ t\in\R\\
u=g(x,t), & x\in\partial\O,\ t\in\R,
\end{array}\right.
\formulaF with $f$ measurable and $g\in C^0(\partial\O\times\R)$.
The boundary condition in \eq{PD} is understood in classical sense:
$u\in C^0(\ol\O\times\R)$ and $u=g$ on $\partial\O\times\R$.

If $L$ is a periodic elliptic operator (as defined in Section
\ref{sec:main}), then $\lambda_{p,D}(-L)$ and
$\pD$ denote respectively the periodic \pe\ and eigenfunction of $-L$ in $\O$, with
Dirichlet boundary conditions. That is, $\lambda_{p,D}(-L)$ is
the unique real number such that the problem
$$\left\{\begin{array}{ll}
-L\pD=\lambda_{p,D}(-L)\pD & \text{in }\O\\
\pD=0 & \text{on }\partial\O
\end{array}\right.$$
admits a solution $\pD$ (unique up to a multiplicative constant)
which is positive in $\O$ and periodic.

The next result is the analogue of \thm{1per}.

\begin{theorem}\label{thm:1perD}
Let $u$ be a bounded solution of \eq{PD}, with $P,\ f,\ g$
periodic in the $m$-th variable, as well as $\O$ if $m\neq N+1$,
and with $c\leq0$. Then, $u$ is periodic in the $m$-th variable.
\end{theorem}

\begin{proof}
We use the same method as in the proof of Lemma \ref{lem:1per},
with $v\equiv1$. As before, it is sufficient to show that the
function
$$\psi(X):=u(X+l_me_m)-u(X)$$ is nonpositive. Assume by
contradiction that $k:=\sup_{\O\times\R}\psi>0$. Let
$\seq{X}=\seq{(x_n,t_n)}$ in $\O\times\R$ be such that
$\psi(X_n)\to k$. We consider the translated
$\psi_n(X):=\psi(X+X_n)$. The problem is that, in principle, these
functions are well defined only at $\{(0\pp0)\}\times\R$. As
$\psi$ is a solution of \eq{PD} with $f\equiv g\equiv0$ and $O$ is
uniformly smooth, parabolic estimates up to the boundary together
with embedding theorem yield $\psi\in UC(\O)$. Hence, there exists
$r>0$ such that $\psi_n>0$ in $B_r\times(-r,r)$ for $n$ large
enough. In particular, the set
$$\mc{R}:=\{r>0\ :\ B_r+\{x_n\}\subset\O\text{ for $n$
large enough}\}$$ is not empty. We claim that $\mc{R}=\R^+$. Let
$r\in\mc{R}$. We know that, for $n$ large enough, the $\psi_n$ are
well defined and uniformly bounded in $B_r\times\R$. Moreover,
again by the estimates up to the boundary, for any $p>1$,
$$\|\partial_t\psi_n\|_{L^p(B_r\times(-r,r))},
\|\partial_i\psi_n\|_{L^p(B_r\times(-r,r))},
\|\partial_{ij}\psi_n\|_{L^p(B_r\times(-r,r))}\leq C,$$ where
$C>0$ is independent of $n$. Therefore, by the compact injection
of $L^p$ in $C^0$, we infer that the $\psi_n$ converge (up to
subsequences) to a bounded solution $\psi_\infty$ of
$$\partial_t \psi_\infty-\t
a_{ij}(x,t)\partial_{ij}\psi_\infty-\t b_i(x,t)\partial_i
\psi_\infty-\t c(x,t)w_\infty=0,\quad  x\in B_r,\ t\in(-r,r),$$
where $\t a_{ij}=\limn a_{ij}(\.+X_n),\ \t b_i=\limn b_i(\.+X_n)$
uniformly in $B_r\times(-r,r)$ and $\t c=\limn c(\.+X_n)$ weakly
in $L^p(B_r\times(-r,r))$. We know that $\psi_\infty$ attains its
maximum value $k$ at $0$ and then the \SMP\ yields
$\psi_\infty(x,t)=k$ for $x\in B_r,\ t\in(-r,0]$ (note that $\t
c\leq0$). As a consequence, for $n$ large, $\psi_n\geq k/2$ in
$B_r\times(-r,0]$ and then, by the uniform continuity, there
exists $\delta>0$ independent of $r$ and $n$ such that $\psi_n>0$
in $B_{r+\delta}\times(-r,0]$. This shows that $\mc{R}=\R^+$. We
then get a contradiction exactly as in the proof of Lemma
\ref{lem:1per}.
\end{proof}

From \thm{1perD} it follows immediatelythe following uniqueness
result (which implies in particular the LP).

\begin{corollary}\label{cor:liouvilleD}
If $\O$ and $P$ are periodic and $c\leq0$ then problem \eq{PD}
admits at most a unique bounded solution.
\end{corollary}

\begin{proof}
Suppose that $u_1,\ u_2$ solve \eq{PD}. Applying \thm{1perD} we
infer that $v:=u_1-u_2$ is periodic and then it has a global
maximum and minimum in $\ol\O$. Since either $\max v\geq0$ or
$\min v\leq0$, the \SMP\ implies that $v$ is constant. But it
vanishes on $\partial\Omega$ and then $v\equiv0$.
\end{proof}

In order to prove the LP when $P=\partial_t-L$ and
$\lambda_D(-L)\geq0$, we will make use of the following
consideration.

\begin{lemma}\label{lem:k}
Let $v_1\in W^{1,\infty}(\O\times\R)$ and $v_2\in C^1(\ol\O\times\R)$
be such that
$$v_1\leq v_2\text{ on }\partial\O\times\R,\qquad
\nabla v_2\in UC(\ol\O\times\R),$$
\formulaI{de<0}
\sup_{\partial\O\times\R}(v_1-v_2+\min(\partial_\nu v_2,0))<0,
\formulaF
\formulaI{Oe}
\fa\e>0,\quad\inf\{v_2(x,t)\ :\
\dist(x,\O^c)>\e,\ t\in\R\}>0.
\formulaF
Then, there exists a positive
constant $k$ such that $kv_2\geq v_1$ in $\O\times\R$.
\end{lemma}

\begin{proof}
Assume by way of contradiction that there exist two sequences
$\seq{x}$ in $\O$ and $\seq{t}$ in $\R$ such that
$nv_2(x_n,t_n)<v_1(x_n,t_n)$. Hence, $\limn v_2(x_n,t_n)=0$ and then
$\dist(x_n,\partial\O)\to0$ by \eq{Oe}. For $n\in\N$, let
$y_n$ denote a projection of $x_n$ on $\partial\O$. We find that
\[\begin{split}
0 &\leq\limn (v_2(y_n,t_n)-v_1(y_n,t_n))=\limn (v_2(x_n,t_n)-v_1(x_n,t_n))\\
&\leq\limn (v_2(x_n,t_n)-nv_2(x_n,t_n))\leq0.\end{split}\]
Therefore, by \eq{de<0},
$$\limsup_{n\to\infty}\partial_\nu v_2(y_n,t_n)<0.$$
As $\nabla
v_2\in UC(\ol\O\times\R)$, we then infer that
\[\begin{split}
\limn\frac{v_1(x_n,t_n)-v_1(y_n,t_n)}{|x_n-y_n|} &\geq \limn
\frac{nv_2(x_n,t_n)-v_2(y_n,t_n)}{|x_n-y_n|}\\
&\geq \limn n\frac{v_2(x_n,t_n)-v_2(y_n,t_n)}{|x_n-y_n|}\\
&= +\infty,
\end{split}\]
which is a contradiction.
\end{proof}

%

We need the following uniform H\"older continuity assumptions:
\formulaI{Holder} a_{ij},b_i,c\in C^\gamma_b(\O), \formulaF
\formulaI{fg} f\in C^{\gamma,\frac\gamma2}_b(\O\times\R),\qquad g\in
C^{2+\gamma,1+\frac\gamma2}_b(\partial\O\times\R), \formulaF for some
$0<\gamma<1$.

\begin{theorem}\label{thm:perD}
Let $P=\partial_t-L$ with coefficients satisfying \eq{Holder} and
let $\O,\ L,\ f,\ g$ be periodic. If $u$ is a bounded solution of
\eq{PD} we have that:
\begin{itemize}

\item[{\rm(i)}] if $\lambda_{p,D}(-L)\geq0$ then $u$ is periodic, with
  the same period as $f,\ g$;

\item[{\rm(ii)}] if $\lambda_{p,D}(-L)=0$ and $f,\ g$ satisfy \eq{fg}
and either $f,g\leq0$ or
$f,g\geq0$ then $u\equiv k\varphi_{p,D}$, for some $k\in\R$, and
$f,g\equiv0$;

\item[{\rm(iii)}] if $\lambda_{p,D}(-L)>0$ and $f,g\equiv0$ then
$u\equiv0$.

\end{itemize}
\end{theorem}

\begin{proof} (i) Fix $m\in\{1\pp N+1\}$ and set
$$\psi(X):=u(X+l_m e_m)-u(X).$$
Let us check that the functions $v_1=\psi$ and $v_2=\pD$ fulfill
the hypotheses of Lemma \ref{lem:k}. Parabolic and elliptic estimates
up to the boundary yield $v_1,v_2\in
C^{2+\gamma,1+\frac\gamma2}_b(\O\times\R)$. Moreover,
$$v_1=v_2=0,\quad\partial_\nu v_2<0,\quad\text{on
}\partial\O\times\R,$$
the last inequality following from the Hopf lemma. Therefore, the
hypotheses of Lemma
\ref{lem:k} are satisfied owing to the periodicity of $\pD$. As a
consequence, there exists $k>0$ such that $k\pD\geq\psi$. Define
$$k^*:=\inf\{k>0\ :\ k\pD\geq\psi\}.$$
Assume by contradiction that $k^*>0$. The function
$w(x,t):=k^*\pD(x)-\psi(x,t)$ is nonnegative by the definition of $k^*$. We
distinguish two different cases.

Case 1:  $w$ satisfies \eq{Oe}.

\noindent If $\sup_{\partial\O\times\R}\partial_\nu w\geq0$ then there exist a
sequence $\seq{Z}$ in $\Z l_1\times\cdots\times\Z l_{N+1}$ and a
sequence $\seq{Y}$ in $\partial\O\times\R$ converging to some
$(y,\tau)\in\partial\O\times\R$ such that
$$\limsup_{n\to\infty}\partial_\nu w(Y_n+Z_n)\geq0.$$
The sequence $w(\.+Z_n)$ converges (up to subsequences) in
$C^{2+\t\gamma,1+\frac{\t\gamma}2}_b(K\times(-r,r))$,
for any $0<\t\gamma<\gamma$, compact set
$K\subset\ol\O$ and $r>0$, to a nonnegative function $w_\infty$ satisfying
$$Pw_\infty\geq0\text{ in }\O\times\R,\qquad w_\infty=0\text{ on
}\partial\O\times\R,\qquad
\partial_\nu w_\infty(y,\tau)\geq0.$$ By Hopf's lemma it follows that
$w_\infty=0$ in $\O\times(-\infty,\tau]$, which is impossible because
$w$ satisfies \eq{Oe}.
This shows that $\sup_{\partial\O\times\R}\partial_\nu w<0$ and then
\eq{de<0} holds with $v_1=\psi$ and $v_2=w$. Therefore, by
Lemma \ref{lem:k} we can find another positive constant $k'$ such that
$k'w\geq\psi$ in $\O\times\R$. That is,
$$\frac {k'}{k'+1}k^*\pD\geq\psi,$$
which contradicts the definition of $k^*$. This case is ruled out.

Case 2:  $w$ does not satisfies \eq{Oe}.

\noindent There exist then a sequence $\seq{Z}$ in $\Z
l_1\times\cdots\times\Z l_{N+1}$ and a sequence $\seq{Y}$ in $\O\times\R$
converging to some $(y,\tau)\in\O\times\R$ such that
$$\limn w(Y_n+Z_n)=0.$$
With usual arguments, we find that (a subsequence of) the sequence
$w(\.+Z_n)$ converges to $0$ locally uniformly in
$\O\times(0,\tau]$. Defining the bounded sequence $(\zeta_h)_{h\in\N}$
as at the end of the proof of \thm{1perD},
we get the following contradiction:
$$\fa h\in\N,\quad\zeta_h-\zeta_{h+1}=k^*\pD(y).$$

In both cases 1 and 2, we have shown that $k^*=0$, that is, $u(X+l_m
e_m)\leq u(X)$. The
converse inequality is obtained in analogous way by replacing $l_m$
with $-l_m$.

(ii) Up to replace $u$ with $-u$, it is not restrictive to assume
that $f,g\leq0$. Hence, $u(x,t)\leq\pD(x)$ for $x\in\partial\O,\
t\in\R$. Note that by parabolic estimates up to the boundary, $u\in
C^{2+\gamma,1+\frac\gamma2}_b(\O\times\R)$.
Applying Lemma \ref{lem:k} with $v_1=u$ and $v_2=\pD$, we
find a positive constant $k$ such that $k\pD\geq u$. Set
$$k^*:=\inf\{k\in\R\ :\ k\pD\geq u\}.$$
The function $w:=k^*\pD-u$ is nonnegative, periodic, by (i), and
satisfies
$$Pw=-f\geq0\quad\text{in }\O\times\R.$$
If $w$ vanishes somewhere in $\O\times\R$ then the parabolic \SMP\ and
the time-periodicity of $w$ yields $w\equiv0$, which concludes the
proof of the statement. Otherwise, for any $x\in\partial\O,\ t\in\R$
such that $w(x,t)=0$, the Hopf lemma yields $\partial_\nu w(x,t)<0$.
Consequently,
$$\fa x\in\partial\O,\ t\in\R,\quad -w(x,t)+\min(\partial_\nu
w(x,t),0)<0.$$
As $\pD$ and $w$ are periodic, we see that the hypotheses of
Lemma \ref{lem:k} are satisfied by $v_1=\pD$ and $v_2=w$ and then
we there exists another positive constant $h$ such that $h
w\geq\pD$. Hence, $(k^*-h^{-1})\pD\geq u$ which contradicts
the definition of $k^*$.

(iii) It is not restrictive to assume that $\sup u\geq0$ (if
not, replace $u$ with $-u$). We proceed exactly as in the proof of
(ii). Now, the constant $k^*$ is nonnegative and then the function
$w:=k^*\pD-u$ satisfies
$$Pw=k^*\lambda_{p,D}(-L)\pD\geq0.$$
Then, as before, we derive $w\equiv0$. From the above expression we
see that $k^*=0$ and then $u\equiv0$.
\end{proof}

\begin{remark}\label{rem:nonHolder}{\rm
\thm{perD} part (i) when $\lambda_{p,D}(-L)>0$ and
part (iii) hold without the additional assumption \eq{Holder}.
In fact, the latter is only used to have the Lipschitz continuity
of solutions required to apply Lemma \ref{lem:k}. But this can be avoided 
by approximating $\O$ by a sequence of domains $(\mc{O}_m)_{m\in\N}$ 
which contain $\O$.
Then, one argues as before, with $\vp_{p,D}$ replaced by the periodic 
principal eigenfunction associated with $\mc{O}_m$. This function is 
strictly positive in $\ol\O$ and is still a supersolution of $-L=0$ provided that
$m$ is large enough, because $\lambda_{p,D}(-L)>0$ (see the proof
of Corollary \ref{cor:!D} below).
This allows one to define the function $w$ -
which does not satisfy \eq{Oe} - 
and obtain the same contradiction as before.}
\end{remark}

\begin{corollary}\label{cor:!D}
If $P=\partial_t-L$, the domain $\O$ and $L$ are periodic,
$\lambda_{p,D}(-L)>0$ and $f\in L^\infty(\O\times\R),\ g\in
W^{2,1}_\infty(\O\times\R)$ \footnote{ $W^{2,1}_\infty$ denotes
the space of functions $u$
  such  that $u,\partial_{i}u,\partial_{ij}u,\partial_t u\in L^\infty$},
then problem
\eq{PD} admits a unique bounded solution.
\end{corollary}

\begin{proof}
Note that, up to replace $f$ with $f-Pg$, it is not restrictive to
assume that $g\equiv0$. As in the proof of Corollary \ref{cor:!}, we find a
solution $u$ as the limit as $n\to\infty$ of solutions $u_n$ of
the problems \formulaI{fredholmD} \left\{\begin{array}{ll}
Pu_n=f(x,t), & x\in\O_n,\ t\in(-n,n)\\
u_n=0, & x\in\partial\O_n,\ t\in(-n,n)\\
u_n(x,-n)=0, & x\in\O_n,
\end{array}\right.
\formulaF where $(\O_n)_{n\in\N}$ is a family of bounded domains
recovering $\O$ (defined below). The proof of the uniform
boundedness  of the $u_n$ is now more delicate, because
$\varphi_{p,D}$ is not bounded from below away from zero and then we
cannot take as sub and supersolution of \eq{fredholmD} the functions
$-k\pD$ and $k\pD$ with $k$ large enough. We overcome this difficulty by
extending $a_{ij},\ b_i,\ c$ to the whole space and by 
considering a domain which is slightly larger than $\O$.
Let $(\mc{O}_m)_{m\in\N}$ be a uniformly smooth family of periodic domains
satisfying
$$\fa m\in\N,\quad\mc{O}_m\supset\mc{O}_{m+1}
\supset\ol\O,\qquad\bigcap_{m\in\N}\mc{O}_m=\O.$$ For any $m\in\N$
let $\lambda_m$ and $\varphi_m$ be the periodic \pe\ and
eigenfunction of $-L$ in $\mc{O}_m$, with Dirichlet boundary
conditions, such that $\|\varphi_m\|_{L^\infty(\mc{O}_m)}=1$. It
follows from the \MP\ that the sequence $(\lambda_m)_{m\in\N}$ is
increasing and bounded from above by $\lambda_{p,D}(-L)$.
Owing to the uniform
smoothness of the $\mc{O}_m$, elliptic estimates up to the
boundary imply that the $\varphi_m$ converge (up to subsequences)
uniformly in $\O$ to a nonnegative periodic solution $\varphi$ of
$-L\varphi=\lambda\varphi$ in $\O$, where
$\lambda=\limn\lambda_m$. Moreover, since for any $\e>0$ there
exists $\delta$ such that
$$\fa m\in\N,\quad\dist(x,\partial\mc{O}_m)\leq\delta\solose
\varphi_m(x)\leq\e$$ (by gradient estimates up to the boundary),
we see that $\varphi$ vanishes on $\partial\O$ and that
$\|\varphi\|_{L^\infty(\O)}=1$. Hence, $\varphi>0$ in $\O$ by the \SMP. 
This shows that $\lambda=\lambda_{p,D}(-L)$.
Thus, there exists $m^*\in\N$ such that $\lambda_{m^*}>0$. The function
$$v(x):=\frac{\|f\|_{L^\infty(\O\times\R)}}{\lambda_{m^*}\min_{\ol\O}
\vp_{m^*}}\vp_{m^*}(x)$$
is the strictly positive supersolution we need to show
that the solutions $u_n$ of \eq{fredholmD} are uniformly bounded
for $n\in\N$. The smooth domains $\O_n$ are defined in such a way that, for 
$n\in\N$, $\O_n\subset B_{n+1}$ and $\O_n\cap B_n$ coincides with the
connected component of $\O\cap B_n$ containing $0$ (which can be assumed 
to belong to $\O$). It is easily seen that for any compact $K\subset\R^N$
there exists $n_0\in\N$ such that $\O\cap K\subset\O_n$ for $n\geq n_0$.
Then, we proceed exactly as in the proof of Corollary \ref{cor:!},
with $B_r$ replaced by $\O_n$.
The uniqueness result follows from \thm{perD} part (iii) and Remark 
\ref{rem:nonHolder}.
\end{proof}

\begin{remark}\rm
If $c\leq0$ then $\lambda_{p,D}(-L)>0$ and then
the existence and uniqueness result of Corollary \ref{cor:!D} applies
(in contrast with the whole space case, cf.~Remark \ref{rem:c}).
This is easily seen by applying the \SMP\ to the periodic principal
eigenfunction $\pD$.
\end{remark}


\subsection{Robin boundary conditions}\label{sec:N}

\def\pN{\varphi_{p,\mc{N}}}

We consider now the Robin problem \formulaI{LN}
\left\{\begin{array}{ll}
Pu=f(x,t), & x\in\O,\ t\in\R\\
\mc{N} u=h(x,t), & x\in\partial\O,\ t\in\R,
\end{array}\right.
\formulaF where
$$\mc{N}u=\alpha(x,t)u+\beta(x,t)\.\nabla u,$$
with $\alpha,\ \beta$ bounded and satisfying
$$\alpha\geq0,\qquad
\inf_{x\in\partial\O}\beta(x)\.\nu(x)>0.$$ 
We always assume in this section that
$$a_{ij},b_i,c\in C_b^{\gamma,\frac\gamma2}(\O\times\R),$$
for some $0<\gamma<1$, and solutions of \eq{LN} are understood in classical sense.
Hence, \eq{LN} admits solution only if $f$ and $h$ satisfy some regularity
conditions, but for our uniqueness results we do not need to impose
them. 

If $P=\partial_t-L$, 
\formulaI{alphabeta}
\alpha=\alpha(x),\ \beta=\beta(x),\qquad
\alpha,\beta\in C^{1+\gamma}_b(\partial\O),
\formulaF
and $\O,\ L,\ \mc{N}$ are periodic then
$\lambda_{p,\mc{N}}$ and $\pN$ denote respectively the periodic
\pe\ and eigenfunction of $-L$ in $\O$, under Robin boundary
conditions. That is, $\lambda_{p,\mc{N}}$ is the unique (real)
number such that the eigenvalue problem
$$\left\{\begin{array}{ll}
-L\pN=\lambda_{p,\mc{N}}\pN & \text{in }\O\\
\mc{N}\pN=0 & \text{on }\partial\O
\end{array}\right.$$
admits a positive periodic solution $\pN$ (unique up to a
multiplicative constant).

The strategy used to prove our results is exactly the same as in
Section \ref{sec:per}, the following lemma being the analogue of
Lemma \ref{lem:1per}. While in the whole space case we used
interior estimates for strong solutions, here we need H\"older
estimates up to the boundary (see \cite{Lady}, \cite{Lie}).

\begin{lemma}\label{lem:1perN}
Assume that $\O$ is periodic and that the operators $P,\ \mc{N}$
and the functions $f,\ h$ are periodic in the $m$-th variable. If
there exists a function $v\in C^{2,1}_b(\O\times\R)$ satisfying
$$\inf_{\O}v>0,\qquad\left\{\begin{array}{ll}
Pv\geq0, & x\in\O,\ t\in\R\\
\mc{N} v\geq0, & x\in\partial\O,\ t\in\R,
\end{array}\right.$$
then any bounded solution of \eq{LN} is periodic in the $m$-th
variable.
\end{lemma}

\begin{proof}
The proof is similar to that of Lemma \ref{lem:1per} and we will
skip some details. But now we translate the functions $\psi,\ v$
and the coefficients of the equation by $Z_n$ instead of $X_n$,
where $\seq{Z}$ is the sequence in $\Z l_1\times\cdots\times\Z
l_{N+1}$ such that
$Y_n:=X_n-Z_n\in[0,l_1)\times\cdots\times[0,l_{N+1})$. Then, the
only situation which is not covered by the arguments in the whole
space is when $w_\infty<k$ in $\O$ and $Y_n$ converges (up to
subsequences) to some
$Y_\infty=(y_\infty,\eta_\infty)\in\partial\O\times[0,l_{N+1}]$.
Let us show that this cannot occur. Let $\alpha^*$ and $\beta^*$
be the limits of (subsequences of) $\alpha(Y_\infty+Z_n)$ and
$\beta(Y_\infty+Z_n)$ respectively. Clearly,
$$\alpha^*\geq0,\qquad
\beta^*\.\nu(y_\infty)>0.$$ Thus, since $w_\infty$ is a solution
of a linear parabolic equation with nonpositive zero order term
achieving a positive maximum at $Y_\infty$, the Hopf lemma yields
$\beta^*\.\nabla w_\infty(Y_\infty)>0$. This is impossible,
because
\begin{equation*}\begin{split}
0 &=\alpha^*\psi_\infty(Y_\infty)+
\beta^*\.\grad\psi_\infty(Y_\infty)\\
&= \alpha^*(w_\infty v_\infty)(Y_\infty)
+\beta^*\.\grad(w_\infty v_\infty)(Y_\infty)\\
&= k(\alpha^*v_\infty(Y_\infty) +\beta^*\.\grad
v_\infty(Y_\infty))+v_\infty(Y_\infty) \beta^*\.\nabla w_\infty(Y_\infty)\\
&\geq v_\infty(Y_\infty) \beta^*\.\nabla w_\infty(Y_\infty)\\
&> 0.
\end{split}\end{equation*}
\end{proof}

Applying Lemma \ref{lem:1perN} with $v\equiv1$ we immediately get

\begin{theorem}\label{thm:1perN}
Let $u$ be a bounded solution of \eq{LN}, with $\O$ periodic, $P,\
\mc{N},\ f,\ h$ periodic in the $m$-th variable and $c\leq0$.
Then, $u$ is periodic in the $m$-th variable.
\end{theorem}

Compare the previous statement with \thm{1perD}, which holds for
domains periodic just in the direction $x_m$. In the case of Robin
boundary conditions, we are only able to deal with domains
periodic in all directions.

\begin{corollary}\label{cor:liouvilleN}
Let $u$ be a bounded solution of \formulaI{P=0N}
\left\{\begin{array}{ll}
Pu=0, & x\in\O,\ t\in\R\\
\mc{N} u=0, & x\in\partial\O,\ t\in\R,
\end{array}\right.
\formulaF with $\O,\ P,\ \mc{N}$ periodic and $c\leq0$. Then, two
possibilities occur:
\begin{itemize}

\item[{\rm 1)}] $c\equiv0,\ \alpha\equiv0$ and $u$ is constant;

\item[{\rm2)}] $\|c\|_{L^\infty(\O)}+\|\alpha\|_{L^\infty(\partial\O)}\neq0$
and $u\equiv0$.

\end{itemize}
\end{corollary}

\begin{proof}
By \thm{1perN} we know that $u$ is periodic in all space/time
directions and then it has global maximum and minimum in
$\O\times\R$. Let $M=\max u=u(x_0,t_0)$. Up to replace $u$ with
$-u$, we can assume without loss of generality that $M\geq0$.
Thus, by the \SMP, either $u=M$ in $\O\times(-\infty,t_0]$, or
$u<M$ in $\O\times(-\infty,t_0]$ and $\ol x\in\partial\O$. The
second case is ruled out because, by Hopf's lemma we would have
$$0<\beta(x_0,t_0)\.\nabla u(x_0,t_0)\leq
\mc{N}u(x_0,t_0)=0.$$ Therefore, $u=M$ in $\O\times(-\infty,t_0]$
and then, by periodicity, in $\O\times\R$. The statement follows.
\end{proof}

\begin{theorem}\label{thm:perN}
Let $P=\partial_t-L$, the functions $\alpha,\ \beta$ satisfy
\eq{alphabeta} and $\O$, $L$, $\mc{N}$, $f$, $h$ be periodic. If $u$
is a bounded solution of \eq{LN} we have that 
\begin{itemize}

\item[{\rm(i)}] if $\lambda_{p,\mc{N}}(-L)\geq0$ then $u$ is periodic;

\item[{\rm(ii)}] if $\lambda_{p,\mc{N}}(-L)=0$ and either $f,h\leq0$ or
$f,h\geq0$ then $u\equiv k\varphi_{p,\mc{N}}$, for some $k\in\R$,
and $f,h\equiv0$;

\item[{\rm(iii)}] if $\lambda_{p,\mc{N}}(-L)>0$ and $f,h\equiv0$ then
$u\equiv0$.

\end{itemize}
\end{theorem}

\begin{proof}
First, we show that
$$\inf_{\O}\pN>0,$$ no matter what the sign of
$\lambda_{p,\mc{N}}(-L)$ is. Indeed, if $\inf_{\O}\pN=0,$ then the
periodicity and the positivity of $\pN$ in $\O$ yield $\pN(y)=0$
for some $y\in\partial\O$. Hence,
$$0=\mc{N}\pN(y)=\beta(y)\.\nabla\pN(y),$$ which contradicts the Hopf
lemma.

(i) The statement follows by applying Lemma \ref{lem:1perN} with
$v=\pN$.

(ii)-(iii) We can argue exactly as in the proof of \thm{per} part
(ii) and (iii). The only different situation is if $w>0$ in
$\O\times\R$ and vanishes at $(x_0,t_0)\in\partial\O\times\R$. In
this case, we get
$$\beta(x_0)\.\nabla
w(x_0,t_0)=\mc{N}w(x_0,t_0)=-\mc{N}u(x_0,t_0)=-h(x_0,t_0)\geq0$$
(we recall that it is not restrictive to assume that $f,h\leq0$).
Once again, this is in contradiction with the Hopf lemma.
\end{proof}

We conclude with the existence and uniqueness result for \eq{LN}.
We assume that
\formulaI{hypex}
f\in C^{\gamma,\frac\gamma2}_b(\O\times\R),
\qquad
h\in C^{2+\gamma,1+\frac\gamma2}_b(\partial\O\times\R),
\formulaF
and we strenghten the regularity condition on $\beta$ in \eq{alphabeta}:
\formulaI{alphabetaex}
\alpha=\alpha(x),\ \beta=\beta(x),\qquad
\alpha\in C^{1+\gamma}_b(\partial\O),
\qquad \beta\in C^{2+\gamma}_b(\partial\O).
\formulaF

\begin{theorem}\label{thm:!N}
If $P=\partial_t-L$, conditions \eq{hypex}-\eq{alphabetaex} hold, $L,\
\mc{N}$ are periodic and 
$\lambda_{p,\mc{N}}(-L)>0$, then 
problem \eq{LN} admits a unique bounded solution $u$. If in
addition $f$ and $h$ are also periodic, then $u$ is periodic.
\end{theorem}

\begin{proof}
From the uniform smoothness of $\O$ it follows that there exists
$\delta>0$ such that each point in $\O^\delta:=\{x\in\O\ :\
\dist(x,\partial\O)<\delta\}$ admits a unique projection $\pi(x)$
on $\partial\O$. Hence, the function $\dist(x,\partial\O)$ is well
defined and smooth in $\O^\delta$. Let $\chi\in C^\infty(\R)$ be a
cut-off function such that $\chi=1$ in $(0,\delta/2)$, $\chi=0$ in
$(\delta,+\infty)$. The function
$$\psi(x,t):=\frac{h(\pi(x),t)}{\beta(\pi(x))
\.\nu(\pi(x))}\dist(x,\partial\O)
\chi(\dist(x,\partial\O))$$ belongs to
$C^{2+\gamma,1+\frac\gamma2}(\O\times\R)$ and 
satisfies $\mc{N}\psi=h$ on $\partial\O$. Therefore, replacing
$f$ by $f-P\psi$, we can take $h\equiv0$ in \eq{LN}. Define
the domains $\seq{\O}$ as in the proof of Corollary
\ref{cor:!D}. Consider a family of cut-off functions
$\seq{\chi}$ uniformly bounded in $C^{2+\gamma,1+\frac\gamma2}_b(\R^N)$
such that, for $n>1$,
$$\chi_n=1\text{ in }B_{n-1},\qquad
\chi_n>0\text{ in } B_n\backslash B_{n-1},\qquad
\chi_n=0\text{ in }\R^N\backslash B_n.$$
Proceeding as
in the proof of Corollary \ref{cor:!}, with $B_R$ replaced by 
$\O_n$ and $\varphi_p$ by
$\pN$ (which has positive infimum), we see that,
as $n\to\infty$, the unique solution of
$$\left\{\begin{array}{ll}
Pu_n=f(x,t), & x\in\O_n,\ t\in(-n,n)\\
(\chi_n\mc{N}+(1-\chi_n))u_n=0, & x\in\partial \O_n,\ t\in(-n,n)\\
u_n(x,-n)=0, & x\in\O_n,
\end{array}\right.$$
converges (up to subsequences) in $C^{2,1}_b(\O\cap K,(-r,r))$, for any compact 
$K\subset\R^N$ and any $r>0$, to a bounded solution of \eq{LN}. 
The uniqueness result is a consequence of \thm{perN}
part (iii).
\end{proof}

Using the Hopf lemma, one can readily check that if $L,\
\mc{N}$ are periodic, $c\leq0$ and $\alpha,\ c$ are not identically
equal to zero, then
$\lambda_{p,\mc{N}}(-L)>0$. Therefore, the result of \thm{!N}
applies in this case.


\def\cprime{$'$} \def\polhk#1{\setbox0=\hbox{#1}{\ooalign{\hidewidth
  \lower1.5ex\hbox{`}\hidewidth\crcr\unhbox0}}}
  \def\cfac#1{\ifmmode\setbox7\hbox{$\accent"5E#1$}\else
  \setbox7\hbox{\accent"5E#1}\penalty 10000\relax\fi\raise 1\ht7
  \hbox{\lower1.15ex\hbox to 1\wd7{\hss\accent"13\hss}}\penalty 10000
  \hskip-1\wd7\penalty 10000\box7}

\end{document}